\documentclass[11pt]{article}
\input{amssymb.sty}
\usepackage{amsmath, amsthm}

\usepackage{hyperref} 
\usepackage{nameref}
\usepackage{makeidx}
\usepackage{enumerate}
\usepackage{stmaryrd}
\usepackage{mathrsfs}
\usepackage{amssymb}
\title{Singularity of the generator subalgebra in $q$-Gaussian algebras}
\author{ \textsc{Chenxu Wen}}

\AtEndDocument{\bigskip{\footnotesize%
  \textsc{Department of Mathematics, Vanderbilt University, 1326 Stevenson Center, Nashville, TN 37240, United States} \par  
  \textit{E-mail address}: \texttt{chenxu.wen@vanderbilt.edu} \par
}}

\date{}

\newtheorem{thm}{Theorem}
\newtheorem*{thm*}{Theorem}
\newtheorem*{cor*}{Corollary}
\newtheorem*{conj*}{Conjecture}
\newtheorem{prop}[thm]{Proposition}
\newtheorem*{prop*}{Proposition}
\newtheorem{cor}[thm]{Corollary}

\newtheorem*{defn*}{Definition}
\newtheorem*{que*}{Question}
\theoremstyle{definition}

\newtheorem{rem}[thm]{Remark}

\def\H{{\mathcal{H}}}
\def\F{{\mathcal{F}}}
\def\G{{\Gamma}}
\begin{document}

\maketitle

\begin{abstract}
Given $-1<q<1$ and a separable real Hilbert space $\H_{\mathbb{R}}$ with dimension no less than 2, we prove that the generator subalgebra in the $q$-Gaussian algebra $\G_q(\H_{\mathbb{R}})$ is singular.
\end{abstract}

\section*{Introduction}
For a real number $-1<q<1$ and for a given separable real Hilbert space $\H_{\mathbb{R}}$, Bo\.{z}ejko and Speicher \cite{BozejkoSpeicher91brownian} introduced the $q$-deformed Fock space $\F_q(\H_{\mathbb{R}})$, in order to construct the $q$-commutation relation
\[\ell^*(e)\ell(f)-q\ell(f)\ell^*(e)=\left\langle e,f\right\rangle Id .\]
They also studied the $q$-Gaussian algebra $\G_q(\H_{\mathbb{R}})$ as the von Neumann algebra generated by the $q$-Gaussian variables $\{\ell(e)+\ell^*(e):e\in \H_{\mathbb{R}}\}$ in \cite{BozejkoSpeicher94CPmaps}. This family of von Neumann algebras has attracted lots of attention since as $q$ varies, we obtain interesting interpolation between different von Neumann algebras: the free group factor when $q=0$, the hyperfinite II$_1$ factor when $q=-1$ and $L^{\infty}(X)$ when $q=1$. It is known that assuming $\dim \H_{\mathbb{R}}\geq 2$, these $\G_q(\H_{\mathbb{R}})$ are II$_1$ factors \cite{ricard2005factoriality}, non-injective \cite{nou04non-injectivity}, strongly solid \cite{avsec2011s-solid} and for small $q$ we actually get back free group factors \cite{GuionnetShlyakhtenko14freetransport}.

One interesting family of subalgebras of $\G_q(\H_{\mathbb{R}})$ is the ones generated by a single $q$-Gaussian variable $\ell(e)+\ell^*(e)$ for some $e\in \H_{\mathbb{R}}$. In \'{E}ric Ricard's proof \cite{ricard2005factoriality} of factoriality for $\G_q(\H_{\mathbb{R}})$, these subalgebras played a fundamental role. Moreover, when $q=0$ those $\ell(e)+\ell^*(e)$ are exactly Voiculescu's semi-circular elements in the free probability theory and those subalgebras are just the generator subalgebra of free group factors. This motivates us to call such subalgebras \textit{the generator subalgebra of $q$-Gaussian algebras}.

The (genuine) generator subalgebras are long known to be singular and even maximal amenable inside the free group factors \cite{popa83maxinjective}. Thus it is natural to ask whether the same hold for generator subalgebras in $q$-Gaussian algebras. The problem is that for $q\neq 0$ and $e_1,e_2\in \H_{\mathbb{R}}$ mutually orthogonal, $\ell(e_1)+\ell^*(e_1)$ and $\ell(e_2)+\ell^*(e_2)$ are no longer freely independent. Another difficulty in dealing with $q$-Fock spaces is that it is hard to control the operator norms of $W(e^{\otimes n})$ (see Preliminaries for the notation). However as we will see, for generator subalgebras coming from orthogonal vectors, the situation is  not so far from the free case and the computation is manageable. The main result of this paper is 
\begin{thm*}
For any $-1<q<1$ and for any separable real Hilbert space $\H_{\mathbb{R}}$ with dimension no less than 2, the generator subalgebras in $q$-Gaussian algebras are singular.
\end{thm*}

The author strongly believes that the generator subalgebras should be maximal amenable inside $q$-Gaussian algebras, but he is not able to prove it. Thus we leave it as a question:
\begin{que*}
Are the generator subalgebras of $\G_q(\H_{\mathbb{R}})$ maximal amenable? Are they disjoint from other maximal amenable subalgebras, in the sense of \cite{wen15radialmasa}?
\end{que*}

Upon finishing writing up this paper, the author learned that Bikram and Mukherjee \cite{BikramMukherjee16generatormasa} independently proved (among other things) similar results for generator subalgebras in $q$-deformed Araki-Woods algebras which contain the main theorem in this paper as a special case. Their approach to show the mixing property of the generator subalgebra is by proving that the left-right measure of the generator subalgebra is Lebesgue absolutely continuous.  The proof essentially boils down to the tracial case and is very similar to ours. Indeed, both the proof of one key ingredient of their paper, \cite[Theorem 5.3]{BikramMukherjee16generatormasa} and the proof of Theorem \ref{Main theorem} in the current paper, are strongly motivated by Ricard's work \cite{ricard2005factoriality}.

\section*{Acknowledgement}

The author is grateful to Stephen Arvsec for bringing the generator subalgebras into his attention during GPOTS 2015 at Purdue. He would also like to thank Purdue University for the warm hospitality.

\section{Preliminaries}
Let $-1<q<1$ be a fixed real number and let $\H_{\mathbb{R}}$ be a separable real Hilbert space with dimension no less than 2. Denote by $\H:=\H_{\mathbb{R}}\otimes_{\mathbb{R}} \mathbb{C}$ the complexification of $\H_{\mathbb{R}}$. Define an inner product on $\bigoplus_{n\geq 0}\H^{\otimes n}$ by 
\[\left\langle e_1 \otimes \cdots\otimes e_n,f_1\otimes\cdots\otimes f_m\right\rangle_q=\delta_n(m)\sum_{\sigma\in S_m} q^{|\sigma|} \left\langle e_1\otimes\cdots\otimes e_n,f_{\sigma(1)}\otimes\cdots\otimes f_{\sigma(m)} \right\rangle, \]
where $S_m$ is the group of permutations on $\{1,\cdots, m\}$, $|\sigma|$ stands for the number of inversions of $\sigma$, $\H^{\otimes 0}=\mathbb{C}\Omega$ is the space spanned by the vacuum vector $\Omega$ and the inner product on the right-hand side is the usual one on the tensor product of Hilbert spaces. The $q$-deformed Fock space $\F_q(\H_{\mathbb{R}})$ is the completion of $(\bigoplus_{n\geq 0}\H^{\otimes n},\left\langle\cdot,\cdot \right\rangle_q)$. We will simply write $\|\cdot\|$ to be the norm induced by this inner product.

For $e\in \H_{\mathbb{R}}$, we define the \textit{left creation operator} $\ell(e)$ on $\F_q(\H_{\mathbb{H}})$ by $\ell(e)(\Omega)=e$ and
\begin{equation}\label{defn of creation}
\ell(e)(e_1\otimes\cdots\otimes e_n)=e\otimes e_1\otimes \cdots\otimes e_n,
\end{equation}
for $n\geq 1$.  $\ell(e)$ is a bounded operator and its adjoint $\ell^*(e)$ is called the \textit{(left) annihilation operator}, which is given by $\ell^*(e)(\Omega)=0$ and 
\begin{equation}\label{defn of annihilation}
\ell^*(e)(e_1\otimes\cdots\otimes e_n)=\sum_{1\leq i\leq n}q^{(i-1)}\left\langle e,e_i\right\rangle e_1\otimes \cdots\otimes \hat{e}_i\otimes \cdots\otimes e_n,
\end{equation}
for $n\geq 1$, where $\hat{e}_i$ means a removed letter. One can define similarly the right creation and annihilation operators.

For $e\in \H_{\mathbb{R}}$, let 
\[W(e)=\ell(e)+\ell^*(e).\]
and let $\G_q(\H_{\mathbb{R}})$ be the von Neumann algebra generated by $\{W(e):e\in \H_{\mathbb{R}}\}$. We call it the \textit{$q$-Gaussian algebra associated with $\H_{\mathbb{R}}$}. It is known \cite{BozejkoSpeicher94CPmaps} that $\G_q(\H_{\mathbb{R}})$ is a finite von Neumann algebra with $\Omega$  a separating and cyclic trace vector. Consequently, each element $x\in \G_q(\H_{\mathbb{R}})$ is uniquely determined by $\xi=x\cdot\Omega\in \F_q(\H_{\mathbb{R}})$ and we write $x=W(\xi)$. This notation is consistent with the above definition for $W(e), e\in \H_{\mathbb{R}}$. Moreover, an easy induction shows that $\G_{q}(\H_{\mathbb{R}})\Omega$ contains all the simple tensors $e_1\otimes \cdots\otimes e_n \in \H^{\otimes n}$.

Here we record two facts that will be used in this paper. 
\begin{itemize}
\item Let $e\in \H$ be a unit vector, then
\begin{equation}\label{l^2 norm estimate}
\|e^{\otimes n}\|^2=[n]_q!,
\end{equation}
where $[k]_q=\dfrac{1-q^k}{1-q}$ and $[n]_q!=[1]_q\cdots [n]_q$.
\item (\textbf{Wick formula}) Let $e_1\otimes \cdots\otimes e_n\in \H^{\otimes n}$, then 
\begin{equation}\label{Wick formula}
\begin{split}
W(e_1\otimes \cdots\otimes e_n)&=\sum_{i=0}^{n}\sum _{\sigma\in S_n/(S_{n-i}\times S_i)} q^{|\sigma|}\ell(e_{\sigma(1)})\cdots\ell(e_{\sigma(n-i)})\\
&\times \ell^*(e_{\sigma(n-i+1)})\cdots\ell^*(e_{\sigma(n)}),
\end{split}
\end{equation}
where $\sigma$ is the representative of the \textit{right} coset of $S_{n-i}\times S_{i}$ in $S_n$ with minimal number of inversions. For each coset such a representative is unique thus the above formula is well-defined.

\end{itemize}

From now on we fix a unit vector $e\in \H_{\mathbb{R}}$ and we call the von Neumann subalgebra $\G_q(\mathbb{R}e)\subset \G_q(\H_{\mathbb{R}})$ a \textit{generator subalgebra}. It is shown by Ricard \cite{ricard2005factoriality} that this gives a maximal abelian subalgebra (masa) of $\G_q(\H_{\mathbb{R}})$, however we will not need this fact in our proof.

Let $T:\H_{\mathbb{R}}\rightarrow\H_{\mathbb{R}}$ be a $\mathbb{R}$-linear contraction. We still denote by $T$ its complexification. Then the first quantization $\F_q(T)$, is the bounded operator on $\F_q(\H_{\mathbb{R}})$ given by 
\[\F_q(T)=Id_{\mathbb{C}\Omega}\oplus\bigoplus_{n\geq 1}T^{\otimes n}.\]

The\textit{ second quantization of} $T$, is the unique unital completely positive map $\G_q(T)$ on $\G_q(\H_{\mathbb{R}})$ given by 
\[\G_q(T)(W(\xi))=W(\F_q(T)(\xi)).\]
In particular, if $T=E_e:\H_{\mathbb{R}}\rightarrow \mathbb{R}e$ is the orthogonal projection, then $\G_q(E_e)$ is the conditional expectation of $\G_q(\H_{\mathbb{R}})$ onto $\G_q(\mathbb{R}e)$.

\section{Singularity of the generator subalgebra $\G_q(\mathbb{R}e)$}

Recall that a von Neumann subalgebra $A\subset M$ is called \textit{singular}, if the normalizer of $A$, defined by $\mathcal{N}_M(A):=\{u\in \mathcal{U}(M):uAu^*=A\}$, is contained in $A$. As is well-known, singularity is closely related to a weakly mixing property: we say that for a finite von Neumann algebra $(M,\tau)$, a subalgebra $A$ is \textit{weakly mixing} in $M$ if there exists a sequence of unitaries $\{u_n\}$ in $A$, such that 
\[\lim_{n\to \infty}\|E_A(au_nb)-E_A(a)u_nE_A(b)\|_2= 0, \forall a,b\in M,\]
where $\|x\|^2_2=\tau(x^*x)$ for any $x\in M$. If the above limit equals $0$ for any sequence of unitaries $\{u_n\}$ in $A$ which converges to 0 weakly, then $A$ is said to be \textit{mixing} in $M$. Clearly for diffuse subalgebras, mixing implies weakly mixing and weakly mixing implies singularity (see \cite{JS08mixingsubalgebra}, \cite{SinclairSmith08masabook}).

We will use this sufficient condition to show the singularity of the generator subalgebra. In order to prove that the generator subalgebra is mixing in the $q$-Gaussian algebra, we will need the following basis criteria from \cite[Theorem 11.4.1]{SinclairSmith08masabook} . The statement is slightly stronger than that in \cite{SinclairSmith08masabook}, but the proof is the same. For convenience of the reader as well as for completeness, we include the proof here.

\begin{prop}\label{basis criteria for singularity}
Let $M$ be a separable finite von Neumann algebra and $A\subset M$ a diffuse subalgebra. Let $Y\subset M$ be a subset whose linear span is $\|\cdot\|_2$-dense in $L^2(M)$ and $\{v_n\}\subset A$ is an orthonormal basis for $L^2(A)$. If 
\[\sum_{n}\|E_A(av_nb)-E_A(a)v_nE_A(b)\|_2^2<\infty ,\]
for all $a,b\in Y$, then $A$ is mixing in $M$. In particular, $A$ is singular in $M$.
\end{prop}

\begin{proof}
Let $\{u_n\}_{n\geq 1}$ be a sequence of unitaries in $A$ which converges to 0 weakly. Let $a,b\in Y$ be arbitrary elements in $Y$, $D:=\sup\{\|a\|,\|b\|\}$ and let $\epsilon>0$ be fixed. By the hypothesis we can find a $k$ such that 
\[\sum_{i>k}\|E_A(av_ib)-E_A(a)v_iE_A(b)\|_2^2<\epsilon^2.\]

Since $u_n\to 0$ weakly, we can choose an $n_0$ such that 
\[|\tau(u_n^*v_i)|\leq\dfrac{ \epsilon}{kD^2} ,\]
for all $1\leq i\leq k$ and for all $n\geq n_0$.

Write $u_n=\sum_i \tau(u_nv_i^*)v_i$, then we have 
\begin{equation}
\begin{split}
\left\|E_A(au_nb)-E_A(a)u_nE_A(b)\right\|_2&=\left\|\sum_i \tau(u_nv_i^*)\left(E_A(av_ib)-E_A(a)v_iE_A(b)\right)\right\|_2\\
&\leq \left\|\sum_{1\leq i\leq k} \tau(u_nv_i^*)\left(E_A(av_ib)-E_A(a)v_iE_A(b)\right)\right\|_2
\\&+\left\|\sum_{i>k} \tau(u_nv_i^*)\left(E_A(av_ib)-E_A(a)v_iE_A(b)\right)\right\|_2.
\end{split}
\end{equation}
For $n\geq n_0$,
\begin{equation}
\begin{split}
&\left\|\sum_{1\leq i\leq k} \tau(u_nv_i^*)\left(E_A(av_ib)-E_A(a)v_iE_A(b)\right)\right\|_2\\
&\leq \sum_{1\leq i\leq k}\dfrac{\epsilon}{kD^2}\left\|E_A(av_ib)-E_A(a)v_iE_A(b)\right\|_2\leq \sum_{1\leq i\leq k}\dfrac{\epsilon}{kD^2}\cdot 2\|a\|\|b\|\|v_i\|_2 \leq 2\epsilon.
\end{split}
\end{equation}
On the other hand, an easy application of Cauchy-Schwarz inequality gives that 
\begin{equation}
\begin{split}
&\left\|\sum_{i>k} \tau(u_nv_i^*)\left(E_A(av_ib)-E_A(a)v_iE_A(b)\right)\right\|_2\\
&\leq \left(\sum_{i>k}|\tau(u_nv_i^*)|^2\right)^{1/2}\left(\sum_{i>k}\left\|E_A(av_nb)-E_A(a)v_nE_A(b)\right\|_2^2\right)^{1/2}\\
&\leq ||u_n||_2\cdot\epsilon\leq\epsilon.
\end{split}
\end{equation}

Therefore,
\[\lim_{n\to\infty}\|E_A(au_nb)-E_A(a)u_nE_A(b)\|_2=0, \forall a,b\in Y.\]
Noticing that $\|u_n\|=1$, a density argument then completes the proof.
\end{proof}

The main result of this paper is 

\begin{thm}\label{Main theorem}
Let $q$ be a real number between $(-1,1)$ and let $\H_{\mathbb{R}}$ be a separable real Hilbert space with dimension greater or equal to $2$. Let $e\in \H_{\mathbb{R}}$ be a unit vector and write $A=\G_q(\mathbb{R}e)$ the generator subalgebra of the $q$-Gaussian algebra $M=\G_q(\H_{\mathbb{R}})$. Suppose that $\{e_i: 0\leq i\leq \dim(\H_{\mathbb{R}})-1\}$ is an orthonormal basis for $\H_{\mathbb{R}}$, where $e_0=e$. We write $E_A$ the conditional expectation from $M$ onto $A$, which can be obtained via the second quantization of $E_e$.

Define 
\[v_0=W(\Omega)=1, v_j=\dfrac{W(e^{\otimes j})}{\|W(e^{\otimes j})\|_2}, \forall j\in \mathbb{N}\]
 and let 
 \[Y=\{W(f_{1}\otimes\cdots \otimes f_{s}):s\geq 0, f_{t}\in \{e_i\}, \forall 1\leq t\leq s\}\subset M.\] 
 Then $\{v_j\}_{j\geq 0}$ is an orthonormal basis for $L^2(A)$ and $span(Y)$ is dense in $L^2(M)$ in the $\|\cdot \|_2$-norm, such that   
\begin{equation}\label{main inequality}
\sum_{j}\|E_A(av_jb)-E_A(a)v_jE_A(b)\|_2^2<\infty ,
\end{equation}
for all $a,b\in Y$. Consequently, $A$ is mixing (thus singular) in $M$.
\end{thm}

\begin{proof}
The statements that $\{v_j\}_{j\geq 0}$ is an orthonormal basis for $L^2(A)$ and that $span(Y)$ is dense in $L^2(M)$ in the $\|\cdot \|_2$-norm are clear from definitions. We just need to show the estimate (\ref{main inequality}).

To this end, first note that if either $a$ or $b$ is from $A$, then (\ref{main inequality}) is trivially true. Indeed, in this case $E_A(av_jb)-E_A(a)v_jE_A(b)=0, \forall j\geq 0$. 

Therefore we can assume that $a,b \in M\ominus A$ are of the form
\[a=W(f_1\otimes\cdots\otimes f_s), b=W(g_1\otimes \cdots\otimes g_t),\]
where $s,t\in \mathbb{N}$, $f_k,g_l\in \{e_i\}_{i\geq 0}, 1\leq k\leq s, 1\leq l \leq t$ and there is some $1\leq k_0\leq s,1\leq l_0\leq t$ such that $f_{k_0},g_{l_0}\in \{e_i\}_{i\geq 1}$.

Moreover, by the Wick formula (\ref{Wick formula}), the form of the annihilation operator and the assumption that $\{e_i:0\leq i\leq \dim(\H_{\mathbb{R}})-1\}$ is an orthonormal set, we can also assume that for each $i\geq 1$, the multiplicity of $e_i$ in the word form of $a$ is the same as in the word form of $b$. 

For simplicity, we will show the conclusion (\ref{main inequality}) when $b$ is of the form
\[b=W(f_1\otimes\cdots f_m\otimes e^{\otimes l}),\]
where $f_1,\cdots, f_m\in \{e_i:i\geq 1\}$. Assume also that the multiplicity of $e$ in the word form of $a$ is $n$ ($n$ may differ from $l$). The other cases are completely similar.

For each $N\geq 0$, let 

\begin{equation}
C_N=\|E_A(av_Nb)-E_A(a)v_NE_A(b)\|_2^2=\dfrac{\left\|\F_q(E_e)\left(aW(e^{\otimes N})(f_1\otimes\cdots f_m\otimes e^{\otimes l})\right)\right\|^2}{\|W(e^{\otimes N})\|^2_2}.
\end{equation}
Our goal is to estimate $C_N$ when $N$ is large.

First we apply the Wick formula to $a$. Note that there are finitely many terms in the expansion, thus it suffices to estimate each term. Also note that each $f_i,1\leq i\leq m$ has to appear in this expansion of $a$ as an annihilation operator, in order to get anything non-zero under $\F_q(E_e)$. Again for simplicity, we consider here only the terms of the form
\begin{equation}
D_N=\dfrac{1}{\|W(e^{\otimes N})\|^2_2}\left\|\F_q(E_e)\left(\ell^n(e)\ell^*(f_1)\cdots\ell^*(f_m)W(e^{\otimes N})(f_1\otimes\cdots f_m\otimes e^{\otimes l})\right)\right\|^2,
\end{equation} 
and the rest cases are almost identical.

Next, we apply Wick formula to $W(e^{\otimes N})$ inside the above expression for $D_N$,
\begin{equation}
\begin{split}
D_N&=\dfrac{1}{\|W(e^{\otimes N})\|^2_2}\left\|\F_q(E_e)\left(\ell^n(e)\ell^*(f_1)\cdots\ell^*(f_m)W(e^{\otimes N})(f_1\otimes\cdots f_m\otimes e^{\otimes l})\right)\right\|^2 \\
&=\dfrac{1}{\|W(e^{\otimes N})\|^2_2}\left\|\F_q(E_e)\left(\ell^n(e)\ell^*(f_1)\cdots\ell^*(f_m)\sum_{i=0}^{N}\sum _{\sigma } q^{|\sigma|}\ell^{N-i}(e)\ell^{*i}(e)(f_1\otimes\cdots f_m\otimes e^{\otimes l})\right)\right\|^2\\
&=\dfrac{1}{\|W(e^{\otimes N})\|^2_2}\left\|\F_q(E_e)\left(\ell^n(e)\ell^*(f_1)\cdots\ell^*(f_m)\sum_{i=0}^{l}\sum _{\sigma} q^{|\sigma|}\ell^{N-i}(e)\ell^{*i}(e)(f_1\otimes\cdots f_m\otimes e^{\otimes l})\right)\right\|^2,
\end{split}
\end{equation}
where the last equality comes from the orthogonality of $\{e_i\}$.

For each $0\leq i\leq l$, we have 
\begin{equation}
\F_q(E_e)\left(\ell^n(e)\ell^*(f_1)\cdots\ell^*(f_m)\ell^{N-i}(e)\ell^{*i}(e)(f_1\otimes\cdots f_m\otimes e^{\otimes l})\right)\in \mathbb{C}e^{\otimes (N+n+l-2i)},
\end{equation}
therefore if we let 
\begin{equation}
D_{N,i}=\dfrac{1}{\|W(e^{\otimes N})\|^2_2}\left\|\F_q(E_e)\left(\ell^n(e)\ell^*(f_1)\cdots\ell^*(f_m)\sum _{\sigma} q^{|\sigma|}\ell^{N-i}(e)\ell^{*i}(e)(f_1\otimes\cdots f_m\otimes e^{\otimes l})\right)\right\|^2,
\end{equation}
then $D_N=\sum_{i=0}^lD_{N,i}$.

Now there is an $M_1=M_1(l,m,q)>0$, such that 
\begin{equation}
D_{N,i}\leq M_1 \cdot \dfrac{|S_N/(S_{N-i}\times S_i)|^2}{\|W(e^{\otimes N})\|^2_2}\left\|\F_q(E_e)\left(\ell^n(e)\ell^*(f_1)\cdots\ell^*(f_m)(e^{\otimes (N-i)}\otimes f_1\otimes\cdots f_m\otimes e^{\otimes (l-i)})\right)\right\|^2.
\end{equation}

For $\ell^n(e)\ell^*(f_1)\cdots\ell^*(f_m)(e^{\otimes (N-i)}\otimes f_1\otimes\cdots f_m\otimes e^{\otimes (l-i)})$ to contribute something non-zero in $\F_q(\mathbb{R}e)$,  each $l^*(f_j)$ has to first pass $e^{\otimes (N-i)}$ then hit $f_1\otimes\cdots\otimes f_m$. The definition for annihilation operators (\ref{defn of annihilation}) then implies that there exists an $M_2=M_2(m,l)>0$, such that 
\begin{equation}
\begin{split}
\left\|\F_q(E_e)\left(\ell^n(e)\ell^*(f_1)\cdots\ell^*(f_m)(e^{\otimes (N-i)}\otimes f_1\otimes\cdots f_m\otimes e^{\otimes (l-i)})\right)\right\|^2\leq M_2 \cdot |q|^{2N}\|e^{\otimes (N+n+l-2i)}\|^2.
\end{split}
\end{equation}

Combining the previous two inequalities, we obtain
\begin{equation}
\begin{split}
D_{N,i}&\leq M_1\cdot M_2\cdot |q|^{2N} \cdot \dfrac{|S_N/(S_{N-i}\times S_i)|^2}{\|W(e^{\otimes N})\|^2_2}\cdot \|e^{\otimes(N+n+l-2i)}\|^2\\
&=M_1\cdot M_2\cdot |q|^{2N}\cdot \left|\dfrac{N!}{(N-i)!i!}\right|^2 \cdot\dfrac{[N+n+l-2i]_q!}{[N]_q!}\\
&\leq M_1\cdot M_2\cdot |q|^{2N} \cdot N^{2i} \cdot\dfrac{[N+n+l-2i]_q!}{[N]_q!}.
\end{split}
\end{equation}
Since $|q|<1$, it is then clear that $\sum_N D_{N,i}<\infty$ which implies that 
\[\sum_N C_N<\infty.\]
The mixing property and the singularity of $A$ in $M$ then follow from Proposition \ref{basis criteria for singularity}.
\end{proof}
\begin{rem}
The reason we use Proposition \ref{basis criteria for singularity} to show the mixing property, instead of proving it directly, is that the basis we choose may be unbounded in the operator (uniform) norm hence the density argument may fail.
\end{rem}
Note that in the proof of the theorem, we never used the fact that $A$ is maximal abelian. In fact, $A$ being singular and abelian implies that it is maximal abelian. Thus our proof recovers Ricard's results from \cite{ricard2005factoriality}.
\begin{cor}
With the same assumptions as in the previous theorem, the generator subalgebra is a masa in the $q$-Gaussian algebra.
\end{cor}
\begin{cor}
For any separable real Hilbert space $\H_{\mathbb{R}}$ with dimension greater or equal to 2, $\G_q(\H_{\mathbb{R}})$ is a II$_1$ factor.
\end{cor}

\small
\bibliographystyle{plain}
\bibliography{ref}

\begin{thebibliography}{10}

\bibitem{avsec2011s-solid}
Stephen Avsec.
\newblock Strong solidity of the q-gaussian algebras for all {$-1< q< 1$}.
\newblock {\em arXiv:1110.4918}, 2011.

\bibitem{BikramMukherjee16generatormasa}
Panchugopal Bikram and Kunal Mukherjee.
\newblock Generator masas in $q$-deformed {A}raki-{W}oods von neumann algebras
  and factoriality.
\newblock {\em arXiv:1606.04752}, 2016.

\bibitem{BozejkoSpeicher91brownian}
Marek Bo{\.z}ejko and Roland Speicher.
\newblock An example of a generalized {B}rownian motion.
\newblock {\em Comm. Math. Phys.}, 137(3):519--531, 1991.

\bibitem{BozejkoSpeicher94CPmaps}
Marek Bo{\.z}ejko and Roland Speicher.
\newblock Completely positive maps on {C}oxeter groups, deformed commutation
  relations, and operator spaces.
\newblock {\em Math. Ann.}, 300(1):97--120, 1994.

\bibitem{GuionnetShlyakhtenko14freetransport}
A.~Guionnet and D.~Shlyakhtenko.
\newblock Free monotone transport.
\newblock {\em Invent. Math.}, 197(3):613--661, 2014.

\bibitem{JS08mixingsubalgebra}
Paul Jolissaint and Yves Stalder.
\newblock Strongly singular {MASA}s and mixing actions in finite von {N}eumann
  algebras.
\newblock {\em Ergodic Theory Dynam. Systems}, 28(6):1861--1878, 2008.

\bibitem{nou04non-injectivity}
Alexandre Nou.
\newblock Non injectivity of the {$q$}-deformed von {N}eumann algebra.
\newblock {\em Math. Ann.}, 330(1):17--38, 2004.

\bibitem{popa83maxinjective}
Sorin Popa.
\newblock Maximal injective subalgebras in factors associated with free groups.
\newblock {\em Adv. Math.}, 50:27--48, 1983.

\bibitem{ricard2005factoriality}
{\'E}ric Ricard.
\newblock Factoriality of q-gaussian von neumann algebras.
\newblock {\em Comm. Math. Phys.}, 257(3):659--665, 2005.

\bibitem{SinclairSmith08masabook}
Allan~M. Sinclair and Roger~R. Smith.
\newblock {\em Finite von {N}eumann algebras and masas}, volume 351 of {\em
  London Mathematical Society Lecture Note Series}.
\newblock Cambridge University Press, Cambridge, 2008.

\bibitem{wen15radialmasa}
Chenxu Wen.
\newblock Maximal amenability and disjointness for the radial masa.
\newblock {\em J. Func. Anal.}, 270(2):787--801, 2016.

\end{thebibliography}

\end{document}